\documentclass[10pt,leqno]{amsart}
\usepackage{graphicx}
\usepackage[T1]{fontenc}

\usepackage{amsmath}
\usepackage{amssymb}
\usepackage{comment}
\usepackage{mathtools}  
\usepackage{xfrac}
\usepackage{float}
\usepackage{verbatim}
\usepackage{faktor}
\usepackage{tikz-cd}
\usepackage{float}
\usepackage{soul}
\usepackage{lipsum}
\usepackage{amsthm}
\usepackage{titleps}
\usepackage{hyperref}
\usepackage{cleveref}
\usepackage{bbold}

\usepackage{wrapfig} 
\usepackage{pifont} 
\usepackage{enumitem} 
\usepackage{graphicx} 
\graphicspath{ {images/} }
\usepackage{pgf,tikz,pgfplots}
\pgfplotsset{compat=1.15}
\usepackage{mathrsfs}
\usetikzlibrary{arrows}
\usepackage{blindtext}
\usepackage{shapepar}
\usepackage{textcomp}
\usepackage{setspace}
\usepackage{multicol}

\newenvironment{namedproof}[1]{%
  \par\pushQED{\qed}%
  \renewcommand{\proofname}{#1}%
  \noindent\textit{\proofname.}\quad\ignorespaces
}{%
  \popQED\par\medskip
}

\newcommand\m[1]{\begin{bmatrix}#1\end{bmatrix}} 

\newcommand\sm[1]{\begin{bsmallmatrix}#1\end{bsmallmatrix}}

\newcommand{\lmfdb}[1]{%
  \href{https://www.lmfdb.org/EllipticCurve/Q/#1}{\texttt{#1}}%
}

\newcommand{\Z}{\mathbb{Z}}
\newcommand{\Q}{\mathbb{Q}}

\newcommand{\GL}{\textnormal{GL}}

\DeclareMathOperator{\Gal}{Gal}

\usepackage{fancyhdr,etoolbox}
\newtheorem{theorem}{Theorem}[]

\newtheorem{lemma}[theorem]{Lemma}

\newtheorem{proposition}[theorem]{Proposition}
\newtheorem{corollary}[theorem]{Corollary}

\theoremstyle{definition}

\newtheorem{definition}[theorem]{Definition}

\theoremstyle{remark}

\newtheorem{remark}[theorem]{Remark}

\hypersetup{ colorlinks=true, linkcolor=black, filecolor=black, urlcolor=black }

\usepackage{lipsum}

\numberwithin{theorem}{section}

\begin{document}
\title{Torsion of $\Q$-curves over number fields of small odd prime degree} 
\author[Initial Surname]{Ivan Novak}
\date{\today}
\address{Faculty of Science, Department of Mathematics, Bijenička cesta 30, Zagreb, Croatia}
\email{ivan.novak@math.hr}
\thanks{The author was supported by the project “Implementation of cutting-edge research and its application as part of the Scientific Center of Excellence for Quantum and Complex Systems, and Representations of Lie Algebras“, PK.1.1.02, European Union, European Regional Development Fund and by the Croatian Science Foundation under the project no. IP-2022-10-5008.
}

\maketitle

\let\thefootnote\relax

\begin{abstract}
 We determine all groups which occur as torsion subgroups of $\Q$-curves defined over number fields of degrees $3$, $5$ and $7$.  In particular, we prove that every torsion subgroup of a $\Q$-curve defined over a number field of degree  $3,5$ or $7$ already occurs as a torsion subgroup of an elliptic curve with rational $j$-invariant.  As the quadratic case has been solved by Le Fourn and Najman, and the case of extensions of prime degree greater than $7$ has been solved by Cremona and Najman, this paper completes the classification of torsion of $\Q$-curves over number fields of prime degree. To solve our problem, we study how images of Galois representations of elliptic curves change upon applying an isogeny. Since Cremona and Najman proved that $\Q$-curves over an odd degree number field $K$ are exactly those which are $K$-isogenous to an elliptic curve with rational $j$-invariant, we can rely on known results about images of Galois representations of rational elliptic curves. 

 As a corollary of our results, we establish that the torsion subgroup of an elliptic curve over a number field $K$ of prime degree which is isogenous to an elliptic curve with rational $j$-invariant is equal to the torsion subgroup of some elliptic curve defined over a degree $p$ number field with rational $j$-invariant. 
\end{abstract} 

\bigskip

\section{Introduction}

\noindent
Let $E/K$ be an elliptic curve, where $K$ is a number field. By the Mordell--Weil theorem, the group $E(K)$ is finitely generated and hence of the form $T \times \mathbb Z^r$ for some nonnegative integer $r$, where $T$ is a finite Abelian group called the \emph{torsion subgroup} of $E$.

A typical question one might ask is to determine all possible torsion subgroups of elliptic curves over a given number field. For elliptic curves over $\Q$, we have the following famous result by Mazur.

\begin{theorem}\label{mazur}\cite{Mazur}
    Let $E/\Q$ be an elliptic curve. Then $E(\Q)_{tors}$ is isomorphic to one of the following $15$ groups:
   \begin{align*}
            &\Z/N_1\Z && \text{ with } N_1=1,\ldots,10,12, \\
            &\Z/2\Z\oplus \Z/2N_2\Z && \text{ with } N_2=1,\ldots, 4.
        \end{align*}
        There exist infinitely many $\overline{\Q}$-isomorphism classes for each such torsion subgroup.
\end{theorem}
In this article, we determine all possible torsion subgroups of $\Q$-curves over number fields of degree $3$, $5$ and $7$. Let us first define $\Q$-curves. For a field $K$, we will denote its absolute Galois group $\Gal(\overline K/K)$ by $G_K$.
\begin{definition}
We say an elliptic curve $E$ over a number field $K$ is a \emph{$\Q$-curve} if it is $\overline K$-isogenous to all its $G_\Q$-conjugates. 
\end{definition}

For example, all elliptic curves isogenous to an elliptic curve with rational $j$-invariant are $\Q$-curves. Indeed, if $E/K$ is $\overline K$-isogenous to an elliptic curve $E'$ with rational $j$-invariant, then there is a twist $E''$ of $E'$ which is defined over $\Q$, and an isogeny $\varphi: E \to E''$. Then there is also an isogeny $\varphi^\sigma: E^\sigma \to (E'')^\sigma$, but $(E'')^\sigma=E''$, and the composition of $\varphi$ and the dual of $\varphi^\sigma$ is an isogeny from $E$ to $E^\sigma$.

Over number fields of odd degree, Cremona and Najman have proven that all non-CM $\Q$-curves arise in this way, with the isogeny even being defined over the base field of the $\Q$-curve.

\begin{proposition}\label{isog to rational j}\cite[Theorem 2.7.]{cremona-najman}
If $E$ is a non-CM $\Q$-curve defined over an odd degree number field $K$, then $E$ is $K$-isogenous to an elliptic curve $E'$ with rational $j$-invariant.    
\end{proposition}

\begin{remark}
    As any isogeny is a composition of a cyclic isogeny and the multiplication--by--$m$ map for some integer $m$, we may assume that the isogeny from the above proposition is cyclic.
\end{remark}

The classification of torsion of base changes of rational elliptic curves, and more generally of elliptic curves with rational $j$-invariant, is known for cubic, quintic and septic fields.

\begin{theorem}\cite[Theorem 1.]{basechanges}\label{basechanges}
    Let $E/\Q$ be a rational elliptic curve, and let $K/\Q$ be a cubic extension. Then $E(K)_{tors}$ is isomorphic to one of the following $20$ groups:
      \begin{align*}
            &\Z/N_1\Z && \text{ with } N_1=1,\ldots,10,12,13,14,18,21, \\
            &\Z/2\Z\oplus \Z/2N_2\Z && \text{ with } N_2=1,2,3,4,7.
        \end{align*}
There exist infinitely many $\overline{\Q}$-isomorphism classes for each such torsion subgroup except for $\Z/21\Z$, for which there is only one elliptic curve with this torsion, which is defined over $\Q(\zeta_9)^+$.
\end{theorem}

\begin{theorem}\cite[Theorem 1]{alvaro-quintic}\label{basechanges5}
    Let $E/\Q$ be a rational elliptic curve, and let $K/\Q$ be a quintic extension. Then $E(K)_{tors}$ is isomorphic to one of the following $17$ groups:
      \begin{align*}
            &\Z/N_1\Z && \text{ with } N_1=1,\ldots,10,11,12,25 \\
            &\Z/2\Z\oplus \Z/2N_2\Z && \text{ with } N_2=1,2,3,4.
        \end{align*}
    There exist infinitely many $\overline{\Q}$-isomorphism classes for each such torsion subgroup except for $\Z/11\Z$, for which there are three elliptic curves with this torsion over $\Q(\zeta_{11})^+$.
\end{theorem}

\begin{theorem}\cite[follows from Proposition 7.7]{basechangesseptic}\label{basechanges7}
   Let $E/\Q$ be a rational elliptic curve, and let $K/\Q$ be a septic extension. Then $E(K)_{tors}$ is isomorphic to one of the $15$ groups from \Cref{mazur}.
\end{theorem}

    





\begin{remark}
    Gužvić proved in \cite[Theorem 5.1.2]{guzvic2021torsion} that elliptic curves with rational $j$-invariant over an odd prime degree number field have the same list of possible torsion subgroups as base changes of rational elliptic curves. 
    
    Over quadratic number fields, there is one exception. Namely, the group $\Z/13\Z$ appears among torsion subgroups of elliptic curves with rational $j$-invariant, and does not appear among torsion subgroups of rational elliptic curves. For example, consider the elliptic curve \hyperlink{https://www.lmfdb.org/EllipticCurve/2.2.17.1/100.1/e/2}{100.1-e2} from the LMFDB.
\end{remark}

The main results of this article are the following.

\begin{theorem}\label{mainresult}
        Let $K/\Q$ be a cubic extension and let $E/K$ be a $\Q$-curve. Then $E(K)_{tors}$ is isomorphic to one of the  $20$ groups appearing in \Cref{basechanges}.
\end{theorem}

\begin{theorem}\label{mainresult5}
    Let $K/\Q$ be a quintic extension and let $E/K$ be a $\Q$-curve. Then $E(K)_{tors}$ is isomorphic to one of the $17$ groups appearing in \Cref{basechanges5}.
\end{theorem}

\begin{theorem}\label{mainresult7}
    Let $K/\Q$ be a septic extension and let $E/K$ be a $\Q$-curve. Then $E(K)_{tors}$ is isomorphic to one of the $15$ groups appearing in \Cref{basechanges7}.
\end{theorem}

\begin{remark}
    In \cite{lefourn}, Le Fourn and Najman found all possible torsion subgroups of $\Q$-curves over quadratic number fields. The new groups that do not occur over the rational numbers are \begin{align*}
        &\Z/N\Z, && \text{ for } N=13,14,15,16,18, \\
        &\Z/2\Z \oplus \Z/2N\Z && \text{ for } N=5,6, \\
        &\Z/3\Z\oplus \Z/3N\Z && \text{ for } N=1,2, \\
        &\Z/4\Z \oplus \Z/4\Z. &&
    \end{align*}  
        In \cite[Theorem 1.3.]{cremona-najman}, Cremona and Najman have shown that when $p>7$ is a prime and $K/\Q$ is an extension of degree $p$, every torsion subgroup $E(K)_{tors}$ of some $\Q$-curve $E/K$ is one of the subgroups already appearing in \Cref{mazur}.
\end{remark}
Combining the previous remark with the results of this paper, the classification of torsion of $\Q$-curves over number fields of prime degree is now completed.

\begin{corollary}\label{korolarglavnog}
    Let $p$ be an odd prime number. If a group $T$ appears as the torsion subgroup of some $\Q$-curve over a number field of degree $p$, then $T$ also appears as the torsion subgroup of some elliptic curve over a number field of degree $p$ with rational $j$-invariant.
\end{corollary}

The above corollary is false for $p=2$. However, a similar statement holds if we replace the class of $\Q$-curves by the class of elliptic curves which are isogenous to an elliptic curve with rational $j$-invariant.

\begin{theorem}\label{pomak-izo}
    Let $p$ be a prime number. If a group $T$ appears as the torsion subgroup of some elliptic curve over a number field of degree $p$ which is $\overline \Q$-isogenous to an elliptic curve with rational $j$-invariant, then $T$ also appears as the torsion subgroup of some elliptic curve over a number field of degree $p$ with rational $j$-invariant.
\end{theorem}




\begin{remark}  Under our definition, CM elliptic curves are also $\Q$-curves. In \cite{CMcase},  Clark, Corn, Rice and Stankewicz classified torsion of CM elliptic curves over number fields of degree at most $13$. They proved that groups which occur as torsion subgroups of CM curves over cubic fields are those from \Cref{mazur} along with $\Z/9\Z$ and $\Z/14\Z$. Over quintic fields, the only new group which appears is $\Z/11\Z$. Over septic fields, no new groups appear. In the rest of this article, we only focus on non-CM elliptic curves.
\end{remark}

\begin{remark}
    Some of the proofs in this paper were aided by the computer algebra system Magma. The auxiliary code is available in the following GitHub repository:
    \begin{center}
        \hyperlink{https://github.com/inova3c/torsion-of-Q-curves}{\texttt{https://github.com/inova3c/torsion-of-Q-curves}}.
    \end{center}

\end{remark}

The paper is structured as follows. In \Cref{generalities}, we list some general helpful results that will be used throughout the paper. In particular, we explicitly spell out the connection between the images of Galois representations of isogenous elliptic curves.

In \Cref{cubiccase}, we prove \Cref{mainresult}, i.e. we solve the problem for cubic extensions. We do this case by case, eliminating each group which occurs for elliptic curves over cubic extensions and does not occur for rational elliptic curves over cubic extensions. Most of the techniques used in the cubic case are also used in the quintic and septic case, which are then discussed in \Cref{case57}. In that case, we do not have a complete list of possible torsion subgroups, but we do have finitely many cases to check as there are finitely many degrees of cyclic isogenies of elliptic curves with rational $j$-invariant. Finally, in \Cref{finalpart}, we conclude the proof of \Cref{pomak-izo}. 

\section*{Acknowledgements}

The author thanks the anonymous referee for carefully reading the article, for numerous helpful comments, and for pointing out an error in a previous version, which has now been corrected. 

\section{General results}\label{generalities}
In this section, we recall some helpful results which will be used throughout the paper. Throughout the section, we will assume that all isogenies are cyclic.

It is a standard fact that every non-CM elliptic curve with rational $j$-invariant is a quadratic twist of an elliptic curve defined over $\Q$.  The following well-known result describes Galois representations of quadratic twists.

\begin{lemma}\label{twistanje}
    Let $E$ be an elliptic curve over a number field $K$. Let $E'$ be a quadratic twist of $E$ by $d \in K^\times$. Let $\psi:G_K \to \{-1, 1\}$ be the quadratic character defined by $\sigma(\sqrt{d})=\psi(\sqrt d) \sqrt d$ for $\sigma \in G_K$. Denote by $\rho_{E, N}$ and $\rho_{E', N}$ the mod $N$ Galois representations of $E$ and $E'$ respectively. Then, for an appropriate choice of bases of $E[N]$ and $E'[N]$, we have $$\rho_{E, N}(\sigma)=\psi(\sigma)\rho_{E'N}(\sigma).$$
\end{lemma}
The proof can be found in \cite[Lemma 5.1]{siksek2019}.

Since non-CM $\Q$-curves over odd degree number fields are isogenous to elliptic curves with rational $j$-invariants, it is very useful to describe what happens to the image of Galois representations after applying isogenies. We start with the following lemma which describes what happens with a point of prime order on the starting elliptic curve. We denote the cyclotomic character modulo $N$ by $\chi_N$. 

\begin{lemma}\cite[Lemma 4.2.]{bourdon-najman}\label{tocka-prostog-reda}
    Let $E_1/K$ be an elliptic curve with a point $P \in E_1(K)$ of prime order $p$ and let $\varphi: E_1 \to E_2$ be an isogeny defined over $K$. Then either $E_2$ has a point of order $p$ defined over $K$ or for some basis of $E_2[p]$ $$\rho_{E_2, p}(\sigma)=\m{\chi_p(\sigma) & y \\
    0 & 1}, \ \text{for all } \sigma \in G_K.$$ 
\end{lemma}

To avoid referring to bases, we may rephrase \Cref{tocka-prostog-reda} as follows.

\begin{corollary}\label{eigen1}
     Let $E_1/F$ be an elliptic curve with a point $P \in E_1(K)$ of prime order $p$ and let $\varphi: E_1 \to E_2$ be an isogeny defined over $K$. Then every $\alpha \in \rho_{E_2, p}(G_K)$ has eigenvalue $1$.
\end{corollary}

Combining  \Cref{twistanje} and \Cref{tocka-prostog-reda}, we obtain the following. 

\begin{corollary}\label{izogenija i eigen+-1}
    Let $E$ be a $\Q$-curve over an odd degree number field $K$ with a point of order $p$ defined over $K$, where $p$ is a prime number. Let $\phi:E\to E'$ be an isogeny, where $E'$ is a curve with rational $j$-invariant, and let $E''$ be a quadratic twist of $E'$ which is defined over $\Q$. Then $E''$ has a $p$-isogeny defined over $K$ and any element of $\rho_{E'',p}(G_K)$ has eigenvalue $1$ or $-1$.
\end{corollary}

The following proposition, which slightly generalizes \cite[Proposition 4.4.]{bourdon-najman},  describes what happens in the composite case.

\begin{proposition}\label{izogenije-generalno}
    Let $E_1/K$ be an elliptic curve with a point $P$ of order $N$ such that $\langle P \rangle$ is fixed by $G_K$. Denote by $H \leq (\Z /N\Z)^\times$ the subgroup $\{k \in (\Z/N\Z)^\times \mid \sigma(P)=kP \text{ for some } \sigma \in G_K\}$. 
    
    Let $\alpha$ be a divisor of $N$ and $\phi:E_1 \to E_2$ the  $\alpha$-isogeny defined over $K$ with kernel $\alpha\cdot P$. Then the image of the mod $N$ representation $\rho_{E_2, N}(G_K)$ is conjugate to a subgroup of the group of all matrices of the form $$\m{ h & \frac{N}{\alpha}u \\ \alpha v & z},$$ where $h \in H$.
\end{proposition}

\begin{proof}
    Let $\{P, Q\}$ be a basis for $E_1[N]$. We claim that $\phi(Q)$ has order $N$. Indeed, if $m\phi(Q)=0$ for $m<N$, then $mQ\in \ker \phi$. Since $\ker \phi$ is generated by a multiple of $P$, one has $mQ+d P=0$ for some integer $d$. However, $P$ and $Q$ are independent, so $m=d\alpha=0$. 

    Now let $P_0$ be a point on $E_1[N]$ such that $\frac{N}{\alpha} P_0=P$. Let $R=\phi(P_0)$ and $S=\phi(Q)$. We claim that $\{R,S\}$ is a basis for $E_2[N]$. Both points are of order $N$, it remains to prove they are independent. If $aR+bS=0$, then $\phi(aP_0+bQ)=0$, so $aP_0+bQ=d\alpha P$ for some number $d$. It follows that $\frac{N}{\alpha}$ divides $a$ since $aP_0$ is of order $N$, and then we have a linear combination of $P$ and $Q$ being equal to $0$. From here, we get $b=0$ and the independence is proven.

    We now compute the action of $\sigma \in G_K$ on both elements of this basis. We know that $$\frac{N}{\alpha}\sigma(R)=\sigma(\phi(P))=\phi(\sigma(P))=h\phi(P)=h \frac{N}{\alpha} R$$ for some $h \in H$,  so $$\sigma(R)=hR+\alpha (mR+nS)=(h+m\alpha)R+n\alpha S$$ for some coefficients $m,n$.

    Furthermore, $\sigma(S)=\sigma(\phi(Q))=\phi(\sigma(Q))$. If we set $\sigma(Q)=uP+vQ$, we get $$\sigma(S)=u\phi(P)+v\phi(Q)=\frac{N}{\alpha}uR+vS.$$
    Note that  $v$ is the bottom right entry of $\rho_{E_1, N}(\sigma)$, and $h$ is the top left entry of $\rho_{E_1, N}(\sigma)$. Their product $hv$ is the determinant of $\rho_{E_1, N}(\sigma)$ which is equal to the cyclotomic character $\chi_N$ by the Weil pairing. This does not depend on the elliptic curve.

    Written as a matrix, we have $$\rho_{E_2, N}(\sigma)=\m{h+m\alpha & \frac{uN}{\alpha} \\ n\alpha & v}.$$

    Taking the determinant, we obtain $$\chi_N(\sigma)=(h+m\alpha)\cdot v.$$ 
    However, we know from before that $\chi_N(\sigma)=hv$. We get $m\alpha v=0$ and hence $m\alpha=0$. Thus, the top left entry of $\rho_{E_2, N}(\sigma)$ is $h$.
\end{proof}

We also spell out the special case when $E_1$ has a point of order $N$ over $K$, not just an $N$-isogeny. This statement is in fact equivalent to \cite[Proposition 4.4]{bourdon-najman}

\begin{corollary}\label{slozeni}
    Let $E_1/K$ be an elliptic curve with a point $P$ of order $N$ defined over $K$. Let $\alpha$ be a divisor of $N$ and $\phi:E_1 \to E_2$ the  $\alpha$-isogeny defined over $K$ with kernel $\alpha\cdot P$. Then the image of the mod $N$ representation $\rho_{E_2, N}(G_K)$ is conjugate to a subgroup of the group of all matrices of the form $$\m{ 1 & \frac{N}{\alpha}u \\ \alpha v & z}.$$
\end{corollary}

The following result is also well-known, but we reprove it for the sake of completion.

\begin{lemma}\label{skok-s-izogenijom}
    Let $K$ be a number field, and let $M$ and $N$ be positive integers. Suppose that for some elliptic curve $E/K$, there are generators $P$ and $Q$ of $E[M^2N]$ such that $[M]P$ and $[MN]Q$ lie in $K$. Then $E$ is $K$-isogenous to an elliptic curve $E_1/K$ such that $\rho_{E_1, M^2N}(G_K)$ is conjugate to a subgroup of the group of matrices of the form $$\m{1+MN\cdot k & * \\ 0 & *}.$$
\end{lemma}
\begin{proof}
    Take an $M$-isogeny $\phi:E\to E_1$ whose kernel is $[MN]Q$. Take a basis of the form $\{\phi(P),S\}$ of $E_1[M^2N]$. In this basis, $\rho_{E_1, M^2N}(G_K)$ has the required shape.
\end{proof}

We will also sometimes use the following result.

\begin{proposition}\cite[Proposition 3.3.]{cremona-najman}\label{nadneparnimnistanovo}
    Let $\ell$ be an odd prime. If a non-CM elliptic curve has no $\ell$-isogenies defined over $\Q$, then it also has no $\ell$-isogenies defined over $K$ for any odd-degree number field $K$, unless $\ell=7$ and $j(E)=2268945/128$, in which case there is a $7$-isogeny defined over the cubic field $\Q(\alpha)$, where $\alpha^3-5\alpha-5=0$.
\end{proposition}

Throughout the article, we will often implicitly use the following basic fact.
\begin{lemma}
    Let $N$ be a positive integer and $E/K$ an elliptic curve. If $F/K$ is a finite extension of fields, then $[\rho_{E,N}(G_K):\rho_{E,N}(G_F)]$ divides $[F:K]$.
\end{lemma}

We remind the reader that possible torsion subgroups of CM elliptic curves are known by results from \cite{CMcase}, so unless stated otherwise we will assume our elliptic curves are non-CM in what follows.

\subsection{Isogeny graphs}

By \Cref{isog to rational j}, any non-CM $\Q$-curve over an odd degree number field $K$ is $K$-isogenous to an elliptic curve with rational $j$-invariant. In some of the proofs, we will use isogeny graphs of elliptic curves. 

Let $K$ be a number field and $E/K$ a non-CM elliptic curve. Consider the graph whose vertices are all elliptic curves over $K$ which are $K$-isogenous to $E$, and for each prime number $p$, two vertices $E_1$ and $E_2$ are connected by an edge with label $p$ if and only if there is a $p$-isogeny from $E_1$ to $E_2$ which is defined over $K$. 

This graph is called the \emph{$K$-isogeny graph} of an elliptic curve $E/K$. We denote it by $\mathcal G(E/K)$. For each prime number $p$, the subgraph of $\mathcal G(E/K)$ induced by vertices which are $p^k$-isogenous to $E$ for some $k\geq 0$ is called the $p$-primary component $\mathcal G_p(E/K)$ of $\mathcal G(E/K)$. For non-CM elliptic curves, $\mathcal G_p(E/K)$ is a tree, for a cycle in the graph would imply the existence of an endomorphism which is not a multiplication-by-$m$ map.

It is well-known that for an odd prime $p$, any elliptic curve $E$ over a number field $K$ has either $0, 1, 2$ or $p+1$ isogenies of degree $p$ defined over $K$. For $p=2$, $E$ has either $0$, $1$ or $3$ isogenies of degree $2$. Furthermore, the following is true.
\begin{lemma}\label{zetap}
    Let $p$ be a prime number. Let $E/K$ be an elliptic curve over a number field $K$ with $p+1$ isogenies of degree $p$ defined over $K$. Then $K$ contains the unique quadratic subfield of the cyclotomic field $\Q(\zeta_p)$. 
\end{lemma}
\begin{proof}
    If $E/K$ has all $p+1$ isogenies of degree $p$ defined over $K$, then $\rho_{E, p}(G_K)$ consists only of scalar matrices, since the action of the Galois group fixes every subgroup of order $p$ of $E[p]\cong (\Z/p\Z)^2$. However, the determinant of a scalar matrix is a square. From the Weil pairing, it follows that the unique quadratic subfield of $\Q(\zeta_p)$ is contained inside $K$.
\end{proof}

\begin{corollary}\label{onlylines}
    If $p$ is an odd prime, $K/\Q$ is a number field of odd degree and $E/K$ is an elliptic curve, the $p$-primary component $\mathcal G_p(E/K)$ is a line graph.
\end{corollary}
\begin{proof}
    The graph $\mathcal G_p(E/K)$ is a tree and hence connected, and \Cref{zetap} implies that any vertex has degree $0,1$ or $2$. The claim follows. 
\end{proof}

\section{Torsion over cubic number fields}\label{cubiccase}

The list of all possible torsion subgroups of eliptic curves over cubic fields is known.

\begin{theorem}\label{cubictorsion}\cite[Theorem A]{cubictorsion}
        Let $K/\Q$ be a cubic extension and $E/K$ be an elliptic curve. Then $E(K)_{tors}$ is isomorphic to one of the following $26$ groups:
        \begin{align*}
            &\Z/N_1\Z && \text{ with } N_1=1,\ldots,16,18,20,21, \\
            &\Z/2\Z\oplus \Z/2N_2\Z && \text{ with } N_2=1,\ldots, 7.
        \end{align*}
        There exist infinitely many $\overline{\Q}$-isomorphism classes for each such torsion subgroup except for $\Z/21\Z$, for which there is only one elliptic curve with this torsion, which is defined over $\Q(\zeta_9)^+$.
\end{theorem}

To prove \Cref{mainresult}, it remains to prove that none of the following six groups appears as torsion of a $\Q$-curve over a cubic field.

 \begin{align*}
            &\Z/N_1\Z && \text{ with } N_1=11,15,16,20 \\
            &\Z/2\Z\oplus \Z/2N_2\Z && \text{ with } N_2=5,6.
\end{align*}

\subsection{Eliminating $\Z/11\Z$} 
\begin{proposition}\label{eliminiraj11}
    The group $\Z/11\Z$ does not occur as a torsion subgroup of a $\Q$-curve over a cubic number field.
\end{proposition}
\begin{proof}
Suppose that some $\Q$-curve $E$ over a cubic field $K$ has a point $P$ of order $11$.
We have the following situation: 
\begin{center}
\begin{tikzcd}
    E/K \arrow[r, rightarrow, "\phi"] & E'/K \arrow[r, rightarrow, "\sim"] & E''/\Q.
\end{tikzcd}
\end{center}
Here $E$ is our starting $\Q$-curve, $\phi$ is a cyclic isogeny defined over $K$, $E'/K$ is an eliptic curve with rational $j$-invariant, $E''$ is an elliptic curve defined over $\Q$ and $\sim$ is an isomorphism from $E'$ to $E''$, i.e. $E''$ is a quadratic twist of $E'$. 

By Corollary \ref{izogenija i eigen+-1}, we know that $E''$ has an $11$-isogeny over $K$. However, by Proposition \ref{nadneparnimnistanovo}, $E''$ also has an $11$-isogeny over $\Q$. 

There are two rational $j$-invariants for which the corresponding elliptic curve $E/\Q$ has an $11$-isogeny over $\Q$. For a fixed $j$-invariant, any two rational elliptic curves with this $j$-invariant are quadratic twists. This means that we can pick any rational elliptic curve with this $j$-invariant and check that its Galois image does not have the properties it would need to have by Corollary \ref{izogenija i eigen+-1}. 

The $j$-invariants for which there is an $11$-isogeny over $\Q$ are $-11^2$ and $-11\cdot 131^3$. Examples of elliptic curves with these $j$-invariants with LMFDB labels  \lmfdb{1089.c1} and \lmfdb{1089.c2}.

For $E''$ being any of these curves, the image of the mod 11 Galois representation $\rho_{E'', 11}(G_\Q)$ has $220$ elements. Since $3 \nmid 220$ and $[\rho_{E'',11}(G_\Q):\rho_{E'',11}(G_K)]$ divides $3$, we must have $\rho_{E'', 11}(G_K)=\rho_{E'', 11}(G_\Q)$. 

Directly inspecting the images $\rho_{E'', 11}(G_\Q)$, there are matrices which have neither $1$ nor $-1$ as an eigenvalue, so we reach a contradiction with Corollary \ref{izogenija i eigen+-1}. 
\end{proof}

\subsection{Eliminating $\Z/15\Z$} \begin{proposition}\label{eliminiraj15}
    The group $\Z/15\Z$ does not occur as a torsion subgroup of a $\Q$-curve over a cubic number field.
\end{proposition}
\begin{proof}
Suppose that some $\Q$-curve $E$ over a cubic field $K$ has a point $P$ of order $15$.
We have the following situation: 
\begin{center}
\begin{tikzcd}
    E/K \arrow[r, rightarrow, "\phi"] & E'/K \arrow[r, rightarrow, "\sim"] & E''/\Q.
\end{tikzcd}
\end{center}
Here $E'$ is an elliptic curve over $K$ with rational $j$-invariant, and $E''$ is a quadratic twist of $E'$ which is defined over $\Q$. 

Then, similarly as in the proof of Proposition \ref{eliminiraj11}, we can conclude that $E''$ has a $3$-isogeny and a $5$-isogeny over $\Q$, so it must have a $15$-isogeny over $\Q$. There are four possible $j$-invariants for which this happens.

As in the proof of Proposition \ref{eliminiraj11}, we can take one specific elliptic curve for each of the $j$-invariants.

We will look at curves from the isogeny class \lmfdb{1600.i} on LMFDB. Note that $-I$ is contained in the mod $15$ image.

By \Cref{twistanje}, we have $\rho_{E', 15} \sim \psi \cdot \rho_{E'', 15}$, where $\psi$ is some quadratic character. Furthermore, any element of $\rho_{E', 15}(G_K)$ has $1$ as an eigenvalue when reduced modulo $3$ and modulo $5$. This means that any element of $\rho_{E'', 15}(G_K)$ either has $1$ as an eigenvalue modulo $3$ and modulo $5$, or it has $-1$ as an eigenvalue modulo $3$ and modulo $5$. Now, for each of the four curves $E''$ from the class 1600.i, we do the following:

\begin{itemize}
    \item List all index $3$ subgroups $H$ of $\rho_{E'', 15}(G_\Q)$.
    \item In each of the subgroups $H$ from the list, find a matrix $A$ such that at least one matrix among $A \pmod 3$,  $A \pmod 5$ doesn't have $1$ as an eigenvalue, and at least one matrix among $A \pmod 3$, $A \pmod 5$ doesn't have $-1$ as an eigenvalue. 
\end{itemize}

The existence a matrix $A \in H$ with this property implies that $H$ cannot be the image $\rho_{E'', 15}(G_K)$.

\begin{enumerate}
    \item \textbf{Case 1600.i1} There are three index $3$ subgroups, so we have three cases. They are listed in order in which they appear in the auxiliary Magma program.
    \begin{enumerate}
        \item The matrices $\sm{2 & 5 \\ 0 & 1}$ and $\sm{1 & 10 \\ 0 & 11}$ are in the image $\rho_{E'', 15}(G_K)$. They both have to be in $\rho_{E', 15}(G_K)$ as their negatives don't have $1$ as an eigenvalue mod $5$. However, their product doesn't have $1$ as an eigenvalue mod $3$. 
        \item The matrices $\sm{12 & 1 \\ 10 & 6}$ and $\sm{6 & 5 \\ 5 & 6}$ are in $\rho_{E'', 15}(G_K)$. They don't have $-1$ as eigenvalues mod $5$, so they themselves have to be in $\rho_{E', 15}(G_K)$. However, their product equals $-I$ mod $3$, which doesn't have $1$ as an eigenvalue mod $3$, so we reach a contradiction. 
        \item Similarly as in the previous case, we end up with $-I$ in $\rho_{E', 15}(G_K)$, which is a contradiction.
    \end{enumerate}
    
    \item \textbf{Case 1600.i2}
    Each of the index $3$ subgroups contains at least one among $\sm{2 & 0 \\ 0 & 11}$ and $\sm{2 & 9 \\ 0 & 11}$. 
    \item \textbf{Case 1600.i3} Each of the index $3$ subgroups contains a matrix of the form $\sm{11 & * \\ 0 & 2}$.
    \item \textbf{Case 1600.i4} The matrix $\sm{11 & 0 \\ 0 & 11}$ is in all index $3$ subgroups.
\end{enumerate}
\end{proof}

\subsection{Eliminating $\Z/16\Z$} 
\begin{proposition}\label{eliminiraj16}
    The group $\Z/16\Z$ does not occur as a torsion subgroup of a $\Q$-curve over a cubic number field.
\end{proposition}
\begin{proof}
    Suppose  that some $\Q$-curve $E$ over a cubic field $K$ has a point $P$ of order $16$. Then $j(E)$ is not rational. 
    
    We know that $E$ is $K$-isogenous to an elliptic curve $E'$ with rational $j$-invariant. Since composing by an isogeny of odd degree does not effect $16$-torsion, we may assume that the isogeny $\phi: E \to E'$ has degree $2^k$ for some positive integer $k$. Then $k\leq 4$ by \cite[Theorem 1.2]{ja}. Suppose that $k>2$. Consider a quadratic twist $E''$ of $E'$ such that $E''$ is defined over $\Q$. Then  $E''$ is $2^k$-isogenous to a quadratic twist $E'''$ of $E$. However, by \cite[Proposition 3.2]{ja}, it follows that for $k>2$, any $2^k$-isogeny of a rational elliptic curve which is defined over a cubic extension is in fact defined over $\Q$. In particular, $E'''$ is also defined over $\Q$, so $j(E''')=j(E)$ is rational, a contradiction.

    Thus, $k\leq 2$. Let $P$ be a $K$-rational point of order $16$ on $E$. We claim that the $K$-rational isogeny $\psi:E \to E_2$ with kernel $\langle P \rangle$ factors through the $2^k$-isogeny $\phi: E \to E'$. Suppose for the sake of contradiction that it does not. If $k=1$, then the dual isogeny $\hat{\phi}:E' \to E$ composed with $\psi: E \to E_2$ yields a cyclic $32$-isogeny of $E'$, a contradiction with \cite[Theorem 1.2]{ja}. If $k=2$ and the isogenies $\phi$ and $ \psi$ are independent, then $E'$ has a $4\cdot 16$-isogeny, a contradiction. If $k=2$ and $\phi$ and $\psi$ are not independent, then $\ker \phi \cap \ker \psi$ has size $2$. We have the following diagram, where each of the edges denotes a $2$-isogeny.
    \[\begin{tikzcd}
	E & E_0 & {E'} \\
	& \bullet \\
	& \bullet \\
	& {E_2}
	\arrow[no head, from=1-1, to=1-2]
	\arrow[no head, from=1-2, to=1-3]
	\arrow[no head, from=1-2, to=2-2]
	\arrow[no head, from=2-2, to=3-2]
	\arrow[no head, from=3-2, to=4-2]
\end{tikzcd}\]
    Since $E'$ and $E_2$ are $16$-isogenous, it follows using the same argument as before that $j(E_2)$ is rational. However, then all the elliptic curves on the path have rational $j$-invariant. Consider a rational quadratic twist $E_0'$ of $E_0$. Then $E_0'$ has two rational $2$-isogenies, so the third of its $2$-isogenies must also be rational. This implies that some twist of $E$ is a rational elliptic curve and hence $E$ has rational $j$-invariant, a contradiction. Thus, $\psi: E \to E_2$ factors through $\phi: E \to E'$. 

    If $k=1$, then there is an $8$-isogeny $E' \to E_2$, so $E_2$ has rational $j$-invariant. However, then $E_2 \to E$ is a $16$-isogeny, so $E$ has rational $j$-invariant, a contradiction. Thus, $k=2$ and we have the following diagram.
    \[\begin{tikzcd}
	E & \bullet & {E'} & \bullet & {E_2}
	\arrow[no head, from=1-1, to=1-2]
	\arrow[no head, from=1-2, to=1-3]
	\arrow[no head, from=1-3, to=1-4]
	\arrow[no head, from=1-4, to=1-5]
\end{tikzcd}\]
By \Cref{slozeni}, it follows that $\rho_{E', 16}(G_K)$ is conjugate to a subgroup of the group of all matrices of the form $\sm{1 & 4x \\ 4y & z}$. Let $E''$  be a quadratic twist of $E'$ which is defined over $\Q$. Then $\rho_{E', 16}(G_K)$ is conjugate to a subgroup of the group of all matrices of the form $\sm{\pm 1 & 4x \\ 4y & z}$.

We see that  $E''$ has a $4$-isogeny over $K$ which is not defined over $\Q$. By \cite[Proposition 3.2]{ja}, it follows that the image of the $2$-adic Galois representation $\rho_{E'', 2^\infty}(G_\Q)$ is conjugate to $H_{20}$, where $H_{20}$ is the corresponding group from the list in \cite{rouse-zureickbrown-2adic}. The group $\rho_{E'', 2^\infty}(G_K)$ is its index $3$ subgroup (it cannot be equal to it, as there would be no $2$-isogenies over $K$). 

However, $H_{20}$ is defined modulo $4$. This implies that $\rho_{E'', 4}(G_K)$ is conjugate to an index $3$ subgroup of the reduction modulo $4$ of $H_{20}$, and hence $\rho_{E'', 2^{\infty}}(G_K)$ is also defined modulo $4$. But then $\rho_{E'', 16}(G_K)$ contains all matrices of the form $I \pmod 4$, which is a contradiction since the matrix $\sm{5 & 0 \\ 0 & 5}$ is not conjugate to a matrix of the form $\sm{\pm 1 & 4x \\ 4y & z}$.
\end{proof}  
\subsection{Eliminating $\Z/2\Z \times \Z/10\Z$ and $\Z/20\Z$}

\begin{proposition}\label{eliminiraj2x10}
    The groups $\Z/2\Z \times \Z/10\Z$ does not occur as a torsion subgroup of $\Q$-curves over cubic number fields.
\end{proposition}
\begin{proof}
    Suppose that $\Z/2\Z \times \Z/10\Z$ does occur for some $\Q$-curve $E$ over a number field $K$. We first use \Cref{skok-s-izogenijom} with $M=2$, $N=5$ to conclude that there is a $2$-isogenous curve $E_1$ such that $\rho_{E_1, 20}(G_K)$ is conjugate to a subgroup of the group of all matrices of the form $\sm{h & * \\
    0  & * }$, where $h \in\{1,11\}$. We see that $E_1$ has a $20$-isogeny $\phi: E_1 \to E_2$ and hence does not have rational $j$-invariant by \cite[Theorem 1.2]{ja}. The curve $E_1$ is also $K$-isogenous to an elliptic curve $E'/K$ with a rational $j$-invariant. Let $\psi: E_1 \to E'$ denote this isogeny.  As before, we may assume $\psi$ is cyclic of degree $2^a\cdot 5^b$ for some $a,b \geq 0$.  

    Since $E'$ has a $5$-isogeny over $K$, it does not have a $4$-isogeny. Hence the degree of $\psi: E_1 \to E'$ is divisible by $2$ and is not divisible by $4$, so it equals $2$ or $10$. It cannot be $50$ by \cite[Theorem 1.2]{ja}.

    Suppose that the degree is $10$. We have the following diagram, with the labels above edges being the degrees. \[\begin{tikzcd}
	{E_1} & {E_3} & {E'}
	\arrow["2", no head, from=1-1, to=1-2]
	\arrow["5", no head, from=1-2, to=1-3]
\end{tikzcd}\]
We claim that $E_3$ also has rational $j$-invariant.  Consider a quadratic twist $E''$ of $E'$ which is defined over $\Q$. By \Cref{nadneparnimnistanovo}, $E''$ has a rational $5$-isogeny. If this $5$-isogeny is not to a twist of $E_3$, then this $5$-isogenous rational elliptic curve has a $50$-isogeny over $K$ to a twist of $E_1$, a contradiction. Thus, $E_3$ is a twist of a rational elliptic curve and hence has rational $j$-invariant. 

Hence, we may assume that the degree of $\psi$ is $2$. If $\psi$ does not factor through $\phi$, then $\psi$ and $\phi$ are independent and we would obtain a $40$-isogeny of $E'$, a contradiction. Thus, $\phi$ factors through $\psi$. However, then $E'$ is isomorphic to the starting curve $E$, since the $20$-isogeny $E_1 \to E_2$ factors through both $E_1 \to E$ and $E_1 \to E'$, and they are both of degree $2$. But then $E$ has rational $j$-invariant, a contradiction.
\end{proof}

\begin{proposition}\label{eliminiraj20}
    The group $\Z/20\Z$ does not occur as a torsion subgroup of a $\Q$-curve over a cubic number field.
\end{proposition}

\begin{proof}
Suppose that the group $\Z/20\Z$ does occur as a torsion subgroup of a $\Q$-curve $E/K$, where $K$ is a cubic field. Then there exists a $2$-isogenous elliptic curve over $K$ with torsion $\Z/2 \Z\times \Z/10\Z$, and we already proved that this is impossible.
\end{proof}

\subsection{Eliminating $\Z/2\Z \times \Z/12\Z$}
\begin{proposition}\label{eliminiraj2x12}
    The group $\Z/2\Z \times \Z/12\Z$ does not occur as a torsion subgroup of a $\Q$-curve over a cubic number field.
\end{proposition}

\begin{proof}
    Suppose that this group does occur for some $\Q$-curve $E$ over a number field $K$. We first use \Cref{skok-s-izogenijom} with $M=2$, $N=6$ to conclude that there is a $2$-isogenous curve $E_1$ such that $\rho_{E_1, 24}(G_K)$ consists (in a suitable basis) of matrices of the form $\sm{h & * \\
    0  & * }$, where $h \in\{1,13\}$. In particular, $E_1$ has a $K$-rational $24$-isogeny. The curve $E_1$ is also $K$-isogenous to an elliptic curve $E'$ with a rational $j$-invariant.  As before, we may assume that this isogeny is cyclic of degree $2^a \cdot 3^b$ for $a,b\geq 0$. Let us factor this isogeny as $E_1 \to E_2$ of degree $3^b$ and $E_2 \to E'$ of degree $2^a$. If $E_2$ has rational $j$-invariant, then $E_2$ has a $3$-isogeny and an $8$-isogeny over $K$, so it has a $24$-isogeny as well. However, this is not a valid isogeny degree over cubic fields by \cite[Theorem 1.2]{ja}. Thus, $E_2$ does not have rational $j$-invariant and hence $a>0$. If $a>2$, then $E'$ has an $8$-isogeny and a $3$-isogeny, which is again impossible for the same reason. We know that $E_2$ has an $8$-isogeny coming from the $8$-isogeny of $E_1$. Consider the path in the $2$-isogeny graph formed by these isogenies. Since the distance between $E'$ and $E_2$ in the graph is $1$ or $2$, the position of $E_2$ with respect to this path is among those marked by $\star$ in the following diagram (note that the diagram does not show the entire $2$-isogeny graph). 
    \[\begin{tikzcd}
	{E_2} & {\star_3} & {\star_5} & \bullet \\
	{\star_1} & {\star_4} \\
	{\star_2}
	\arrow[no head, from=1-1, to=1-2]
	\arrow[no head, from=1-1, to=2-1]
	\arrow[no head, from=1-2, to=1-3]
	\arrow[no head, from=1-3, to=1-4]
	\arrow[no head, from=2-1, to=3-1]
	\arrow[no head, from=2-2, to=1-2]
\end{tikzcd}\]
    If $E'$ is in the position of $\star_1$,  $\star_2$ or $\star_4$, it has an $8$-isogeny, a contradiction. In other cases, $E'$ has a $4$-isogeny to an elliptic curve with non-rational $j$-invariant and an independent $2$-isogeny, so $\rho_{E', 2}(G_K)$ is trivial. Since $E'$ has a $4$-isogeny to an elliptic curve with non-rational $j$-invariant, it follows from \cite[Proposition 3.2]{ja} that some quadratic twist $E''$ of $E'$ which is defined over $\Q$ has $2$-adic Galois image $\rho_{E'', 2^\infty}(G_\Q)$ equal to the group $H_{20}$ from the database of Rouse and Zureick--Brown in \cite{rouse-zureickbrown-2adic}. Reducing modulo $4$, this is a group of order $12$ generated by $\sm{1 & 1 \\ 3 & 0}$ and $\sm{3 & 3 \\ 0 & 1}$. The group $\rho_{E'',4}(G_K)$ is a subgroup of order $4$ which is contained inside the group of uppertriangular matrices. There is only one such subgroup. However, if this subgroup is equal to $\rho_{E'',4}(G_K)$, then $\rho_{E'',2}(G_K)$ is not trivial, because $\rho_{E'', 4}(G_K)$ contains $\sm{3 & 3 \\ 0 & 1}$.  This is a contradiction since then $E''$ (and hence also $E'$) does not have all $2$-isogenies defined over $K$.
\end{proof}

\section{Torsion over number fields of degree $5$ or $7$}\label{case57}

\subsection{Possible prime power orders of points}

\begin{lemma}\label{twist-dualno}
    Let $E$ be a $\Q$-curve over a number field $K$ of prime degree $p\geq 5$. Let $\ell \neq p$ be an odd prime. Suppose that $E$ has a point of order $\ell \neq p$ defined over $K$.
    
      Then there exists an elliptic curve $E_1$ which is $K$-isogenous to $E$ such that $E_1$ has rational $j$-invariant and a point of order $\ell$.
\end{lemma}

\begin{proof}
    Let $E$ be a $\Q$-curve over $K$. Let $E'$ be the $K$-isogenous elliptic curve with rational $j$-invariant. We can factor the isogeny $E \to E'$ by writing it down as $E \to E_1 \to E'$, where the isogeny $E \to E_1$ has degree coprime to $\ell$ and the isogeny $E_1 \to E'$ has order $\ell^k$ for some $k\geq 0$. Then $E_1$ is a $\Q$-curve isogenous to $E$ which also has a point of order $\ell$ over $K$.

    If $\ell^k=1$, then $E_1$ is isomorphic to $E$ and hence has rational $j$-invariant. Otherwise, let $E''$ be the quadratic twist of $E'$ which is defined over $\Q$. Furthermore, let $E'''$ be the corresponding quadratic twist of $E_1$, such that we have the following diagram (the vertical lines are quadratic twists, and the horizontal lines are $\ell^k$-isogenies).

    \[
\begin{tikzpicture}
    \matrix (m) [matrix of math nodes, row sep=3em, column sep=6em, 
                 text height=1.5ex, text depth=0.25ex]
    {
        E_1 & E' \\
        E''' & E'' \\
    };
    
    \path[-] (m-1-1) edge node[above] {$\ell^k$} (m-1-2);
    \path[-] (m-2-1) edge node[above] {$\ell^k$} (m-2-2);
    
    \path[-] (m-1-1) edge node[left] {twist} (m-2-1);
    \path[-] (m-1-2) edge node[right] {twist} (m-2-2);

\end{tikzpicture}
\]

Note that $p$ does not divide $\ell \cdot (\ell-1)$ unless $p=5$ and $\ell=11$, so in those cases $\rho_{E'', \ell^k}(G_\Q)=\rho_{E'', \ell^k}(G_K)$ and it follows that any $\ell^k$-isogeny of $E''$ defined over $K$ is in fact already defined over $\Q$.

Now suppose that $\ell=11$ and $p=5$. Then $E''$ is a rational elliptic curve with an $11$-isogeny over a quintic extension, so by \Cref{nadneparnimnistanovo} it must have a rational $11$-isogeny. However, then $j(E'')=j(E')$ belongs to the set $\{-11^2, -11\cdot 131\}$. The elliptic curves with these $j$-invariants have only one $11$-isogeny (the rational one) over quintic extensions of $\Q$. They have no isogenies of odd prime order over $\Q$, so their isogeny class does not increase when taking quintic extensions. It follows that they are the only $\Q$-curves over quintic extensions with a point of order $11$.
\end{proof}

The above result implies that any possible prime order of a point already appears among elliptic curves with rational $j$-invariant.

In particular, by \Cref{basechanges5} there are no points of prime order greater than $11$ over degree $5$ extensions, and by \cite[Proposition 7.7]{basechanges}, there are no points of prime order greater than $7$ over degree $7$ extensions. Since the points of order $11$ have just been discussed, we may now only focus on powers of $2,3,5,7$.

\begin{lemma}
    A $\Q$-curve over a number field of degree $5$ or $7$ does not have a point of order $16$ or $27$.
\end{lemma}
\begin{proof}Note that $5$ and $7$ do not divide the orders of $\GL_2(\Z/2\Z)$ and $\GL_2(\Z/3\Z)$.

Suppose that there is a $\Q$-curve $E$ over a degree $5$ or $7$ extension with a $16$-isogeny or a $27$-isogeny and let $E'$ be an isogenous curve with rational $j$-invariant, $E''$ the quadratic twist of $E'$ which is rational and $E'''$ the corresponding twist of $E$, as in the diagram in the proof of \Cref{twist-dualno}.  Without loss of generality we may assume that the isogeny $E \to E'$ is a $2^a 3^b$-isogeny for some $a,b \geq 0$. 

However, the corresponding $2^a3^b$-isogeny $E'' \to E'''$ is defined over $\Q$ since $5$ and $7$ do not divide $\GL_2(\Z/2\Z)$ or $\GL_2(Z/3\Z)$, so $E'''$ is a rational elliptic curve and its twist $E$ has rational $j$-invariant. The claim now follows by \Cref{basechanges5} and \Cref{basechanges7}.
\end{proof}

\begin{lemma}\label{divisibleby25}
    If $E/\Q$ is an elliptic curve with a $5$-isogeny, the index of the image of its $5$-adic Galois representation in $\GL_2(\Z_5)$ is not divisible by $25$.
\end{lemma}
\begin{proof}
    This follows immediately from \cite[Theorem 2.]{greenberg}, since the index of a Sylow pro-$5$ subgroup of $\GL_2(\Z_5)$ is not divisible by $25$.
\end{proof}

\begin{lemma}\label{5izogenije}
    There are no $\Q$-curves over degree $5$ number fields with a point of order $125$, and no $\Q$-curves over degree $7$ number fields with a point of order $25$.
\end{lemma}
\begin{proof}
    Suppose that $E$ is a $\Q$-curve defined over a quintic extension of $\Q$ with a point of order $125$. We know that $E$ is $K$-isogenous to an elliptic curve with rational $j$-invariant $E'$. Without  loss of generality, we may assume that the isogeny degree is $5^a$ for some $a\geq 1$. If $a>3$, then $E'$ has a $625$-isogeny over a quintic extension, which is impossible by \cite[Theorem 1.3]{ja}. Thus $a\leq 3$. The elliptic curve $E$ also has a $125$-isogeny defined over $K$ whose kernel is the point of order $125$ on $E$. Since the $5$-isogeny subgraph is a line by \Cref{onlylines}, it follows that the isogeny $E \to E'$ either factors through the $125$-isogeny whose kernel is the torsion point or the isogenies are independent. If they are independent, then $E'$ has a $5^{3+a}$-isogeny, which is again a contradiction. Thus, $E \to E'$ factors through the $125$-isogeny. Then, by \Cref{slozeni}, $\rho_{E', 125}(G_K)$ is conjugate to a subgroup of the group of all matrices of the form $$\m{1 & 5^a \cdot x \\
    5^{3-a}\cdot y & *}$$ for some $a \in \{0,1,2,3\}$. 

    If $E''$ is a quadratic twist of $E'$ which is a rational elliptic curve, then $\rho_{E'', 125}(G_K)$ is conjugate to a subgroup of the group of  all matrices of the form $$\m{\pm 1 & 5^a \cdot x \\
    5^{3-a}\cdot y & *}$$ for some $a \in \{0,1,2,3\}$. 

This means that $\rho_{E'', 125}(G_K)$ has at most $2\cdot 25\cdot 4 \cdot 125=2^3\cdot 5^5$ elements. It follows that the index of $\rho_{E'', 125}(G_K)$ is divisible by $\#\GL_2(\Z/125\Z)/(2^3\cdot 5^5)=2^2\cdot 3 \cdot 5^4$, so it is divisible by $5^4$. It follows that the index of $\rho_{E'', 125}(G_\Q)$ is divisible by $5^3$. Since $E''$ has a $5$-isogeny over $\Q$, \Cref{divisibleby25} implies that $\rho_{E'', 5^\infty}(G_\Q)$ is not divisible by $25$, so we reach a contradiction.

Now suppose that $E$ is defined over a degree $7$ extension $K$ and has a point of order $25$. Let $E'$ be the isogenous curve which has rational $j$-invariant and let $5^a$ be the isogeny degree. We have $a\leq 2$ since $125$ is not a possible isogeny degree of an elliptic curve with rational $j$-invariant. 

The $5$-isogeny graph is a line so we can again conclude that the isogeny $E \to E'$ either factors through the isogeny whose kernel is a point of order $25$ or they are independent. If they are independent, then $E'$ has a $125$-isogeny over an extension of degree $7$, a contradiction, so $E \to E'$ factors through the isogeny whose kernel is the torsion point. If we let $E''$ be the quadratic twist of $E'$ which is defined over $\Q$, we know by \Cref{slozeni} that $\rho_{E'', 25}(G_K)$ is contained in the subgroup of the group of all matrices of the form $$\m{\pm 1 & 5^a \cdot x \\
    5^{2-a}\cdot y & *}$$ for some $a \in \{0,1,2\}$. Since $7 \nmid |\GL_2(\Z/5\Z)|$, it follows that  $\rho_{E'', 25}(G_K)=\rho_{E'', 25}(G_\Q)$. We can now compute that the index of $\rho_{E'', 25}(G_\Q)$ is at least $3 \cdot 2^2 \cdot 5^2$, which is again divisible by $25$, a contradiction with \Cref{divisibleby25}.
\end{proof}

\begin{lemma}
    There are no $\Q$-curves over number fields of degree $5$ or $7$ with a point of order $49$.
\end{lemma}

\begin{proof}
    Let $E$ be a $\Q$-curve over a degree $5$ extension $K$ with a point of order $49$, $E'$ be the $K$-isogenous curve with rational $j$-invariant and $E''$ the twist of $E'$ which is rational. We may assume that the isogeny $E \to E'$ has degree $7^a$ for $a\geq 1$. Since $7^2$ is not a possible isogeny degree of elliptic curves with rational $j$-invariant over quintic fields, it follows that $a=1$ and that the isogeny $E \to E'$ factors through the isogeny whose kernel is the point of order $49$.  By \Cref{izogenije-generalno}, $E'$ has two independent $7$-isogenies over $K$. Then $E''$ also has two independent $7$-isogenies. Since $5 \nmid |\GL_2(\Z/7\Z)|$, they are also defined over $\Q$. However, there then exists a $\Q$-isogenous elliptic curve with a $49$-isogeny over $\Q$, which is a contradiction. 

    Now suppose that $E$ is a $\Q$-curve over a degree $7$ extension $K$ with a point of order $49$, $E'$ the $K$-isogenous curve with rational $j$-invariant and $E''$ the twist of $E'$ which is rational. We may assume that the isogeny $E \to E'$ has degree $7^a$ for $a\geq 1$. Since $7^3$ is not a possible isogeny degree of an elliptic curve with rational $j$-invariant over a septic field, it follows that the isogeny $E \to E'$ factors through the isogeny whose kernel is the point of order $49$. Furthermore, since a twist $E''$ of $E'$ has a rational $7$-isogeny, we may assume that we have the following diagram. 
    \[\begin{tikzcd}
	E & {E'} & {E_2}
	\arrow[no head, from=1-1, to=1-2]
	\arrow[no head, from=1-2, to=1-3]
\end{tikzcd}\]
Here $E'$ and $E_2$ both have rational $j$-invariant. 
    
    by \Cref{izogenije-generalno}, it follows that for the quadratic twist $E''$ of $E'$ the group $\rho_{E'', 49}(G_K)$ is a subgroup of the group of all matrices of the form $$\m{\pm 1 & 7 \cdot x \\
    7\cdot y & *}$$ Since $E''$ has a $7$-isogeny over $\Q$, \cite[Theorem 3.9]{lombardo} implies that that the $7$-adic representation of $E''$ over $\Q$ is defined modulo $7$, i.e. it contains all matrices congruent to $I$ modulo $7$.

    To summarize, $\rho_{E'', 49}(G_\Q)$, once a suitable basis has been chosen, has a subgroup of index $7$ contained in the group of all matrices of the form $\sm{\pm 1 & 7 \cdot x \\
    7\cdot y & *}$, which has order $7^3\cdot 2\cdot 6$ and $\rho_{E'', 49}(G_\Q)$ contains the group of all matrices congruent to $I$ modulo $7$, which has order $7^4$. This means that the order of  $\rho_{E'', 49}(G_K)$ must be divisible by $7^3$. 
    
    Without loss of generality, taking another quadratic twist if needed, we may assume that $\rho_{E'', 49}(G_K)$ also contains $-I$. We know that $\rho_{E''. 49}(G_K)$ has surjective determinant in $(\Z/7\Z)^\times$. The preimages under $\det: \rho_{E'', 7}(G_K) \to (\Z/7\Z)^\times$ of all elements of $(\Z/7\Z)\times$ have the same number of elements, and this number is divisible by $2$ since $\det(-I)=1$.  It follows that $\rho_{E'', 7}(G_K)$ and consequently $\rho_{E'', 49}(G_K)$, has order divisible by $12$. But then  $\#\rho_{E'', 49}(G_K)$ is divisible by $7^3 \cdot 12$, and $7\cdot\#\rho_{E'', 49}(G_K)=\#\rho_{E'', 49}(G_\Q)$ divides $7^4\cdot 12$. It follows that we must have equality everywhere, so $\rho_{E'', 49}(G_K)$ is \emph{equal} to the group of all matrices of the form $\sm{\pm 1 & 7\cdot x \\ 7\cdot y & *}.$ Since $\rho_{E'', 49}(G_\Q)$ contains all matrices of this form and all matrices congruent to $I \pmod 7$, it follows that $\rho_{E'', 49}(G_\Q)$ is the group of all matrices of the form $\sm{\pm 1+7\cdot z & 7\cdot x \\ 7\cdot y & *}$. But then $E''$ has two independent rational $7$-isogenies, a contradiction. 
\end{proof}

\subsection{Possible odd orders of points}

\begin{lemma}
    There are no $\Q$-curves over number fields of degree $5$ or $7$ with a point of order $35$.
\end{lemma}

\begin{proof}
    If $E$ is such a $\Q$-curve, then there exists an elliptic curve with rational $j$-invariant in the isogeny class with a $35$-isogeny over an extension of degree $5$ or $7$, which is impossible by \cite[Theorem 1.3]{ja}.
\end{proof}

\begin{lemma}
    There are no $\Q$-curves over number fields of degree $5$ or $7$ with a point of order $15$ or $21$.
\end{lemma}
\begin{proof}
    Suppose that there is a $\Q$-curve $E$ over a number field $K$ of degree $5$ or $7$ with a point of order $15$ over an extension of degree $5$ or $7$. Then an elliptic curve with rational $j$-invariant in the isogeny class has a quadratic twist which has a rational $15$-isogeny or a $21$-isogeny. We can now copy the proof of \Cref{eliminiraj15} to conclude that we must do the following for each of the four curves $E''$ from the class 1600.i:

\begin{itemize}
    \item List all index $5$ or $7$ subgroups $H$ of $\rho_{E'', 15}(G_\Q)$.
    \item In each of the subgroups $H$ from the list, find a matrix $A$ such that at least one matrix among $A \pmod 3$,  $A \pmod 5$ doesn't have $1$ as an eigenvalue, and at least one matrix among $A \pmod 3$, $A \pmod 5$ doesn't have $-1$ as an eigenvalue. 
\end{itemize}

The existence of a matrix $A \in H$ with this property implies that $H$ cannot be the image $\rho_{E'', 15}(G_K)$.

In fact, there are no index $7$ subgroups as $7$ doesn't divide the order of $\GL_2(\Z/15\Z)$.

\begin{enumerate}
    \item \textbf{Case 1600.i1} There are five index $5$ subgroups. Each of them contains a matrix of the form $\sm{2 & 3k+2 \\ 0 & 1}$, and all of them contain $\sm{11 & 0 \\ 10 & 1}$. They both have to be in $\rho_{E', 15}(G_K)$ as their negatives do not have $1$ as an eigenvalue mod $5$. However, their product does not have $1$ as an eigenvalue mod $3$.

    \item \textbf{Case 1600.i2}
    Each of the index $5$ subgroups contains a matrix of the form $\sm{2 & * \\ 0 & 11}$. Those matrices do not have eigenvalue $1$ mod $3$, and their negatives do not have eigenvalue $1$ mod $5$.
    \item \textbf{Case 1600.i3} Each of the index $5$ subgroups contains a matrix of the form $\sm{6 & 3k+2 \\ 10 & 7}$ and the matrix $\sm{11 & 5 \\ 0 &1}$, whose negatives do not have $1$ as an eigenvalue mod $5$, and whose product does not have $1$ as an eigenvalue mod $3$.
    \item \textbf{Case 1600.i4} The matrix $\sm{11 & 0 \\ 0 & 11}$ is in all index $5$ subgroups.
\end{enumerate}

We can use the same reasoning for points of order $21$. We just need to consider elliptic curves in the isogeny class 1296.f. There are four such curves. In this case, there are no index $5$ subgroups, so we need to look at index $7$ subgroups.

For each of the four curves in the isogeny class, the mod $21$ image has seven subgroups of index $7$. Each of those contains an element which does not have neither $1$ nor $-1$ as an eigenvalue modulo $7$, which completes the proof.
\end{proof}

\subsection{Combining the $2^k$-torsion with the odd part}

Note that for any $n>2$, the entire $n$-torsion $E[n]$ cannot be defined over an odd degree number field $K$, as $K(\zeta_n)$ is contained in the field $K(E[n])$ and $\zeta_n$ is not contained in $K$.

As there are no points of order $16$ this means that the possible $2^k$-torsion over $K$ is either trivial or of the form $\Z/2^k\Z$ or $\Z/2\Z\oplus \Z/2^k\Z$ for $k \in \{1,2,3\}$.

It remains to combine possible $2$-torsion with the possible groups of points of odd order.

\begin{lemma}
    There are no $\Q$-curves over number fields of degree $5$ or $7$ with a point of order $14$.
\end{lemma}
\begin{proof}
      If $E$ was such a $\Q$-curve, then there would be a rational elliptic curve $E''$ with a $14$-isogeny over an extension of degree $5$ or $7$, which is impossible by \cite[Theorem 1.3]{ja}, unless $E''$, and hence also $E$, has CM. However, by \cite[5.5 and 5.7]{CMcase}, this is also not possible.
\end{proof}

\begin{lemma}
    There are no $\Q$-curves over number fields of degree $5$ or $7$ with a point of order $18$ or $24$, or independent points of order $2$ and order $12$.
\end{lemma}
\begin{proof}
    If $E$ is a $\Q$-curve over a number field of degree $5$ or $7$ with the mentioned torsion group, we can assume that there is a $2^k\cdot 3^\ell$-isogeny to an elliptic curve $E'$ with rational $j$-invariant. Let $E''/\Q$ be the rational twist of $E'$, and let $E'''$ be the corresponding twist of $E$. Then there is a $2^k3^\ell$-isogeny $E'' \to E'''$ defined over $K$. Since $5$ and $7$ do not divide $\#\GL_2(\Z/3\Z)|$ and $\#\GL_2(\Z/2\Z)$, it follows that this isogeny is in fact defined over $\Q$, so $E'''$ is also rational. We conclude that $E$ has rational $j$-invariant. By the results in \cite{alvaro-quintic} for $K$ of degree $5$, and \cite{basechanges} for $K$ of degree $7$, and using \cite[Theorem 5.1.2]{guzvic2021torsion}, we reach a contradiction.
\end{proof}

\begin{lemma}
    There are no $\Q$-curves over number fields of degree $7$ with a point of order $20$ or independent points of order $2$ and $10.$
\end{lemma}
\begin{proof}
    We apply the same argument as in the proof of the previous lemma but with an isogeny of degree $2^k5^\ell$, noting that $7$ does not divide $\#\GL_2(\Z/5\Z)$ and $\#\GL_2(\Z/2\Z)$. We obtain an elliptic curve with rational $j$-invariant and it cannot have a point of order $20$ or independent points of order $2$ and $10$ by \cite[Theorem 5.1.2]{guzvic2021torsion}. 
\end{proof}

The proof of \Cref{mainresult7} is now complete, i.e.\ the possible torsion subgroups of elliptic curves over septic fields have all been determined.. 

\begin{lemma}
    There are no $\Q$-curves over number fields of degree $5$ with independent points of order $2$ and $10$, or a point of order $20$. 
\end{lemma}
\begin{proof}
   Let $E$ be a $\Q$-curve over a number field $K$ of degree $5$ with two independent points of order $2$ and $10$ or with a point of order $20$. In both of these cases, $E$ is $K$-isogenous to an elliptic curve $E_1$ with a $20$-isogeny. The elliptic curve $E_1$ is $2^a \cdot 5^b$-isogenous to an elliptic curve $E'$ with rational $j$-invariant. We may factor this $2^a\cdot 5^b$-isogeny as a $5^b$-isogeny $E\to E_2$ and a $2^a$-isogeny $E_2 \to E'$. Since $5$ does not divide $\#\GL_2(\Z/2\Z)$, it follows that $E_2$ also has rational $j$-invariant, so we may assume that $a=0$ and that $E_1$ is $5^b$-isogenous to $E'$.

    However, $E'$ has a $4$-isogeny over $K$ since $E_1$ has a $4$-isogeny. Since it also has a $5$-isogeny, it has a $20$-isogeny. But $20$ is not a possible isogeny degree of an elliptic curve with rational $j$-invariant, so we have reached a contradiction. 
\end{proof}
\begin{lemma}
    There are no $\Q$-curves over number fields of degree $5$ with a point of order $50$.
\end{lemma}

\begin{proof}
    Let $E$ be a $\Q$-curve over a quintic number field $K$ with a point $P$ of order $50$. Then there is an isogenous elliptic curve $E'$ with rational $j$-invariant. As in the previous proof, we may assume that the isogeny $E \to E'$ is a $5^b$-isogeny for some $b\geq 1$. If $b\geq 3$, then $E'$ is an elliptic curve with rational $j$-invariant and a $250$-isogeny over a quintic extension, which is impossible by \cite[Theorem 1.3]{ja}. If the isogeny $E \to E'$ does not factor through the $50$-isogeny whose kernel is the point of order $50$, then $E'$ again has a $2\cdot 5^{2+b}$-isogeny which is impossible. Thus, $b\leq 2$ and the isogeny $E \to E'$ factors through the isogeny whose kernel is the point of order $50$.

    Suppose that $b=2$. We have the following diagram. 
    \[\begin{tikzcd}
	E & \star & {E'} & \bullet
	\arrow[no head, from=1-1, to=1-2]
	\arrow[no head, from=1-2, to=1-3]
	\arrow[no head, from=1-3, to=1-4]
\end{tikzcd}\]
By \Cref{nadneparnimnistanovo}, we know that some quadratic twist $E''$ of $E'$ has a rational $5$-isogeny, so either the elliptic curve marked with $\star$ or the elliptic curve marked with $\bullet$ has rational $j$-invariant. If the elliptic curve marked with $\bullet$ has rational $j$-invariant, then it has a $250$-isogeny over $K$, a contradiction. Thus, the elliptic curve marked with $\star$ has rational $j$-invariant. Thus we may actually assume that $b=1$ and treat the elliptic curve marked with $\star$ as $E'$.

By \Cref{izogenije-generalno}, $\rho_{E', 50}(G_K)$ is contained in the subgroup of the group of all matrices of the form $\sm{1 & 5\cdot k \\
    10 \cdot j & z}$. It follows that $\#\rho_{E',25}(G_K)$ divides $5^3\cdot 4$. Let $E''$ be a quadratic twist of $E'$ which is defined over $\Q$. Then $\rho_{E'', 25}(G_K)$ has cardinality dividing $5^4\cdot 8$. Also, $E''$ has a rational $5$-isogeny by \Cref{nadneparnimnistanovo}. By \cite[Theorem 2]{greenberg}, if none of the curves in the $\Q$-isogeny class of $E''$ has two independent $5$-isogenies, then $\rho_{E'', 5^\infty}(G_\Q)$ contains a Sylow pro-$p$ subgroup. However, the Sylow pro-$p$ subgroup when viewed modulo $25$ has $5^5$ elements, which does not divide $5^4\cdot 8$. Thus, one of the curves in the $\Q$-isogeny class of $E''$ has two independent rational $5$-isogenies. Since $E''$ also has a rational $2$-isogeny coming from $E$, it follows that the isogeny class a rational elliptic curve with a $50$-isogeny, which is a contradiction.
\end{proof}

This completes the proof of \Cref{mainresult5}, i.e.\ the possible torsion subgroups over quintic fields have all been determined.

\section{Quadratic torsion of elliptic curves isogenous to those with rational $j$-invariant}\label{finalpart}

In this section, we prove \Cref{pomak-izo}.

\begin{proposition}\label{eliminiraj1418}
    Let $E$ be an elliptic curve over a quadratic number field $K$ which is $\overline \Q$-isogenous to an elliptic curve with rational $j$-invariant. Then $E(K)_{tors}$ is not isomorphic to $\Z/14\Z$ or $\Z/18\Z$.
\end{proposition}
\begin{proof}
    Suppose that $E(K)$ has torsion $\Z/14\Z$. Let $E'$ be the elliptic curve with rational $j$-invariant isogenous to $E$. Applying a quadratic twist, one may assume that $E'$ itself is defined over $K$. Then, by \cite[Lemma A.4]{cremona-najman}, there is a quadratic twist $E''$ of $E'$ such that there is an isogeny $E \to E''$ defined over $K$ itself. By \Cref{izogenije-generalno}, it follows that $E''$ has a $14$-isogeny over $K$, which is not a possible isogeny degree of elliptic curves with rational $j$-invariant over quadratic fields by \cite{Borna}.

    Suppose that $E(K)$ has torsion $\Z/18\Z$. By the results in \cite[Section 4.6.]{dujella-bosman}, it follows that $E$ is $2$-isogenous to its Galois conjugate. If $E$ is isogenous to an elliptic curve with rational $j$-invariant then, using notation from \cite[Appendix]{cremona-najman}, the $j$-invariant of $E$ belongs to a rational class $\mathcal Q$. However, the degree of an isogeny from $E$ to its conjugate then has to be a square by \cite[Corollary A.12]{cremona-najman} and it cannot be $2$.
\end{proof}

\begin{namedproof}{Proof of \Cref{pomak-izo}}
    For $p>2$ this follows from \Cref{korolarglavnog}, as elliptic curves over odd prime degree number fields $K$ are exactly those isogenous to elliptic curves with rational $j$-invariant.

    For $p=2$, we know from the list of all possible torsion subgroups over quadratic fields by the results from Kamienny \cite{kamienny}, and Kenku and Momose \cite{kenku-momose}. Najman \cite[Theorem 2]{basechanges} found the list of all torsion subgroups of rational elliptic curves over quadratic extensions, and Gužvić \cite[Theorem 5.1.2.]{guzvic2021torsion} found the list of all torsion subgroups of elliptic curves with rational $j$-invariant over quadratic extensions.  
    
    It follows from these results that $\Z/14\Z$ and $\Z/18\Z$ are the only torsion subgroups that appear for some elliptic curves over quadratic fields and do not appear for elliptic curves with rational $j$-invariant. By \Cref{eliminiraj1418}, they also do not appear for elliptic curves isogenous to those with rational $j$-invariant, which completes the proof.
\end{namedproof}

\bibliographystyle{abbrv}
\bibliography{lit}

\end{document}